\documentclass[12pt]{amsart}
\usepackage{fullpage, amsmath, amsthm,amsfonts,amssymb,stmaryrd, mathrsfs}
\usepackage{hyperref}

\usepackage[pdftex]{color}
\usepackage[all]{xy}
\usepackage{caption}
\usepackage{graphicx}

\numberwithin{equation}{subsection}
\newtheorem{theorem}[subsection]{Theorem}

\newtheorem{corollary}[subsection]{Corollary}
\newtheorem{lemma}[subsection]{Lemma}
\newtheorem{proposition}[subsection]{Proposition}

\theoremstyle{definition}

\newtheorem{definition}[subsection]{Definition}
\newtheorem{example}[subsection]{Example}

\newtheorem{notation}[subsection]{Notation}
\newtheorem{question}[subsection]{Question}
\newtheorem{remark}[subsection]{Remark}

\def\calC{\mathcal{C}}

\def\calW{\mathcal{W}}

\def\AAA{\mathbb{A}}

\def\CC{\mathbb{C}}
\def\FF{\mathbb{F}}
\def\GG{\mathbb{G}}

\def\NN{\mathbb{N}}

\def\QQ{\mathbb{Q}}

\def\ZZ{\mathbb{Z}}

\DeclareMathOperator{\Gal}{Gal}

\newcommand{\Qp}{\QQ_p}
\newcommand{\rig}{\mathrm{rig}}

\newcommand{\wt}{\mathrm{wt}}

\newcommand{\Zp}{\ZZ_p}

\newcommand{\Tr}{\mathrm{Tr}}

\begin{document}

\title{Newton Slopes for Artin-Schreier-Witt  Towers}
\author{Christopher Davis}
\address{Christopher Davis, University of California, Irvine, Department of
Mathematics, 340 Rowland Hall, Irvine, CA 92697}
\email{davis@math.uci.edu}
\author{Daqing Wan}
\address{Daqing Wan, University of California, Irvine, Department of
Mathematics, 340 Rowland Hall, Irvine, CA 92697}
\email{dwan@math.uci.edu}
\author{Liang Xiao}
\address{Liang Xiao, University of Connecticut, Storrs, Department of
Mathematics, 196 Auditorium Road, Unit 3009, Storrs, CT 06269}
\email{liang.xiao@uconn.edu}

\date{\today}

\begin{abstract}
We fix a monic polynomial $f(x) \in \FF_q[x]$ over a finite field and consider the Artin-Schreier-Witt tower defined by $f(x)$; this is a tower of curves $\cdots \to C_m \to C_{m-1} \to \cdots \to C_0 =\AAA^1$, with total Galois group $\ZZ_p$.
We study the Newton slopes of zeta functions of this tower of curves. This reduces to the study of the Newton slopes of L-functions associated to characters of the Galois group of this tower.  We prove that, when the conductor of the character is large enough, the Newton slopes of the L-function form arithmetic progressions which are independent of the conductor of the character.  As a corollary, we obtain a result on the behavior of the slopes of the eigencurve associated to the Artin-Schreier-Witt tower, analogous to the result of Buzzard and Kilford.
\end{abstract}

\maketitle

\setcounter{tocdepth}{1}
\tableofcontents

\section{Introduction}

We fix a prime number $p$.
Let $\FF_q$ be a finite extension of $\FF_p$ of degree $a$ so that $q =p ^a$.
For an element $b \in \overline{\FF}_p$, let $\hat b$ denote its Teichm\"uller lift in $\ZZ_p^{\mathrm{ur}}$.
We fix a monic polynomial $ f(x) = x^d +  b_{d-1}x^{d-1} + \cdots +  b_0 \in \FF_q[x]$ whose degree $d$ is not divisible by $p$. Set $b_d := 1$. Let $\hat f(x)$ denote the polynomial $ x^d + \hat b_{d-1}x^{d-1} + \cdots + \hat b_0 \in \ZZ_q[x]$.
The \emph{Artin-Schreier-Witt} tower associated to $f(x)$ is the sequence of curves $C_m$ over $\FF_q$ defined by the following equations:
\[
C_m: \quad \underline y_m^F - \underline y_m =  \sum_{i=0}^d(b_ix^i, 0, 0,\dots)
\]
where $\underline y_m = (y_m^{(1)}, y_m^{(2)}, \dots)$ are viewed as Witt vectors of length $m$, and $\bullet^F$ means raising each Witt coordinate to the $p$th power.
In explicit terms, this means that $C_1$ is the usual Artin-Schreier curve given by $y^p - y = f(x)$, and $C_2$ is the curve above $C_1$ given by an additional equation (over $\FF_q$)
\[
z^p - z + \frac{y^{p^2}-y^p - (y^p - y)^p}{p} =  \frac{\hat f^\sigma(x^p) - \hat f(x)^p}{p},
\]
where $\hat f^\sigma(x) : = x^d + \hat b_{d-1}^p x^{d-1} + \dots +\hat b_0^p$.

It is clear that the Artin-Schreier-Witt tower is a tower of smooth affine curves $\cdots \to C_m \to C_{m-1} \to \cdots \to C_0 :=\AAA^1_{\FF_q}$, forming a tower of Galois covers of $\AAA^1$ with total Galois group $\ZZ_p$. This tower is totally ramified at $\infty$. Thus, each curve $C_m$ has only one point at $\infty$, 
which is $\FF_q$-rational  and smooth.  It is well known that the zeta function of the affine curve $C_m$ is 
\[
Z(C_m,s) = \exp \left(\sum_{k\geq 1} \frac{s^k}{k} \cdot \# C_m(\FF_{q^k}) \right) =\frac{P(C_m,s)}{1-qs},  
\]
where $P(C_m, s) \in 1+s\ZZ[s]$ is a polynomial of degree $2g(C_m)$, pure of $q$-weight $1$, and $g(C_m)$ denotes the genus of $C_m$.   

A natural interesting problem, in the spirit of Iwasawa theory,  is to understand the $q$-adic Newton slopes of this sequence $P(C_m,s)$ of polynomials,  especially their 
stable properties as $m\rightarrow \infty$. This seems to be a difficult problem for a general tower of curves, and in fact it is not clear if one should expect any 
stable property for the $q$-adic Newton slopes. For the Artin-Schreier-Witt tower of curves 
considered in this paper, we discover a surprisingly strong stability property for the $q$-adic Newton slopes. 

Our problem for the zeta functions easily reduces to the corresponding problem for the L-functions attached to 
the tower of curves. 
In this paper, all characters are assumed to be continuous.
For a finite character $\chi: \ZZ_p \to \CC_p^\times$, we put $\pi_\chi = \chi(1)-1$.
Let $m_\chi$ be the nonnegative integer so that the image of $\chi$ has cardinality $p^{m_\chi}$; we call $p^{m_\chi}$ the \emph{conductor} of $\chi$. Then, when $\chi$ is nontrivial,  $\QQ_p(\pi_\chi)$ is a finite totally ramified (cyclotomic) extension of $\Qp$ of degree $(p-1)p^{m_\chi - 1}$, and $\pi_\chi$ is a uniformizer.
Such a character $\chi$ defines an L-function $L(\chi, s)$ over $\AAA^1_{\FF_q}$ given by
\[
L(\chi, s) = L_f(\chi, s) = \prod_{x \in |\AAA^1|}
\frac{1}{1-\chi\bigg(\Tr_{\QQ_{q^{\deg(x)}}/\QQ_p}\left(\hat f(\hat x)\right)\bigg)s^{\deg(x)}}
\in 1+ s\ZZ_p[\pi_\chi][\![ s]\!],
\]
where $ |\AAA^1|$ denotes the set of closed points of $\AAA^1_{\FF_q}$ and $\hat{x}$ denotes the Teichm\"uller lift of  any of the 
conjugate geometric points in the closed point $x$. 
The L-function $L(\chi, s)$ is known to be a polynomial of degree $p^{m_\chi-1} d-1$ if $\chi$ is non-trivial, see Theorem 1.3 in 
\cite{liu-wei}.   
The L-function of the trivial character is given by $L(1,s) = 1/(1-qs)$,  which is just the zeta function of $\AAA^1_{\FF_q}$. 
The zeta functions of the curves in the Galois tower admit the following decompositions:
\[
Z(C_m,s) = \prod_{\chi, 0\leq m_{\chi}\leq m} L(\chi, s), \quad \ P(C_m,s) = \prod_{\chi, 1\leq m_{\chi}\leq m} L(\chi, s). \]
Hence the study of the polynomial $P(C_m,s)$ reduces to the study of $L(\chi,s)$ for various nontrivial finite characters $\chi$.

In this paper we study certain periodicity behavior of the Newton polygon of the L-function $L(\chi,s)$.  We first explain our conventions on Newton polygons.

\begin{notation}
Let $R$ be a ring with valuation and $\varpi$ an element with positive valuation.
Let $v_\varpi(\cdot)$ denote the \emph{$\varpi$-adic valuation} on $R$ normalized so that $v_\varpi(\varpi) =1$.
Then the \emph{$\varpi$-adic Newton polygon} of a polynomial or a power series $1+ a_1 s + a_2 s^2 + \cdots$ with coefficients in $R$, is the lower convex hull of the set of points $(i, v_\varpi(a_i))$ for $i=0, 1, \dots$ (put $a_0=1$).  The \emph{slopes} of such a polygon are the slopes of each of its width~1 segments, counted with multiplicity and put in increasing order.
Clearly, changing the choice of $\varpi$ in $R$ results in rescaling the slopes.
\end{notation}

If $\chi_1$ and $\chi_2$ are two characters with the same conductor $p^{m_{\chi_1}}=p^{m_{\chi_2}}=p^m>1$, then their L-functions $L(\chi_1,s)$ 
and $L(\chi_2,s)$ are Galois conjugate polynomials over $\QQ(\zeta_{p^m})$ and hence have the same $p$-adic Newton polygon.  
Our main result is the following

\begin{theorem}
\label{T:main theorem}
Let $m_0$ be the minimal positive integer such that $p^{m_0-1} \geq \frac{a(d-1)^2}{8d}$ and let $0<\alpha_1, \dots, \alpha_{dp^{m_0-1}-1}<1$ denote the slopes of the $q$-adic Newton polygon of $L(\chi_0,s)$ for  a finite character $\chi_0: \ZZ_p \to \CC_p^\times$ with $m_{\chi_0} = m_0$. 
Then, for every finite character $\chi:\ZZ_p \to \CC_p^\times$ with $m_\chi \geq m_0$, the $q$-adic Newton polygon of $L(\chi,s)$ has slopes

\[
\bigcup_{i=0}^{p^{m_{\chi}- m_0}-1} \big\{\frac{i}{p^{m_{\chi}-m_0}}, \frac{\alpha_1+i}{p^{m_{\chi}-m_0}}, 
\dots, \frac{\alpha_{dp^{m_0-1}-1}+i}{p^{m_{\chi}-m_0}}\big\} - \{0\},
\]
In other words, the $q$-adic Newton slopes of $L(\chi, s)$ form a union of $dp^{m_0-1}$ arithmetic progressions, with increment 
$p^{m_0-m_{\chi}}$.
\end{theorem}
In short, this theorem says that the Newton slopes of the L-function for an Artin-Schreier-Witt tower enjoy a certain periodicity property.
From this, one can easily deduce a nice description for the Newton slopes of the zeta function $P(C_m,s)$, as $m \to \infty$. 
For an integer $m\geq 1$, write 
$$P(m,s) =\prod_{m_{\chi}=m} L(\chi, s) \in 1+s\ZZ[s],$$
which is a polynomial of degree $(p-1)p^{m-1}(p^{m-1}d-1)$. Then, 
$$P(C_m, s) =\prod_{k=1}^m P(k, s) \in 1+s\ZZ[s]$$
is a polynomial of degree 
$$2g(C_m) = \sum_{k=1}^m (p-1)p^{k-1}(p^{k-1}d-1) = (p-1)\left(d\frac{p^{2m}-1}{p^2-1} -\frac{p^m-1}{p-1}\right).$$
With the notations of the previous theorem, we deduce 

\begin{corollary} \label{C:Newton slopes cor} For each integer $m\geq m_0$, the $q$-adic Newton polygon of $P(m, s)$ has slopes 
\[
\bigcup_{i=0}^{p^{m- m_0}-1} \big\{\frac{i}{p^{m-m_0}}, \frac{\alpha_1+i}{p^{m-m_0}}, \dots, \frac{\alpha_{dp^{m_0-1}-1}+i}{p^{m-m_0}}\big\} - \{0\},
\]
each counted with multiplicity $(p-1)p^{m-1}$.
\end{corollary}

\begin{example} \label{ordinary example}
We say that $L(\chi, s)$ is \emph{ordinary} if its Newton polygon is equal to its Hodge polygon; see for example \cite[\S1.2]{wan3} or \cite[\S3]{fu-wan}.   The Newton slopes are easily determined in this case.  For instance, when $p \equiv 1 \pmod d$, the example after Corollary~3.4 in \cite{wan3} says that for any character $\chi$ of order~$p$, the L-function $L(\chi, s)$ is ordinary.  Theorem 2.9 in \cite{fu-wan} then implies that $L(\chi, s)$ is ordinary for all non-trivial $\chi$.  One deduces that for each integer $m\geq 1$, the $q$-adic Newton polygon of $P(m, s)$ has slopes 
\[
\bigcup_{i=0}^{p^{m- 1}-1} \big\{\frac{i}{p^{m-1}}, \frac{\frac{1}{d}+i}{p^{m-1}}, \dots, \frac{\frac{d-1}{d}+i}{p^{m-1}}\big\} - \{0\} = \big\{\frac 1{dp^{m-1}},
\frac 2{dp^{m-1}}, \dots, \frac {dp^{m-1}-1}{dp^{m-1}}
\big\},
\]
each counted with multiplicity $(p-1)p^{m-1}$.
Corollary~\ref{C:Newton slopes cor} is an analogue of this result which does not require $L(\chi, s)$ to be ordinary.
\end{example}

Both the proof of Theorem~\ref{T:main theorem} and the proof of \cite[Theorem~2.9]{fu-wan} referenced in Example~\ref{ordinary example}, rely on the study of a $T$-adic version of the characteristic function, which is an entire version of the L-function. This $T$-adic characteristic function interpolates the characteristic functions associated to all finite characters.  The key observation we make in this paper is that the obvious lower bound of the $T^{a(p-1)}$-adic Newton polygon of the $T$-adic characteristic function  agrees, at  all points in an arithmetic progression, with the $\pi_\chi^{a(p-1)}$-adic polygon for the characteristic function for the specialization at a finite  character $\chi$.  Therefore, the two Newton polygons agree at these points. As a by-product, we prove that a strong form of the $T$-adic Riemann hypothesis holds for 
the $T$-adic L-function in \cite{fu-wan}, namely, the splitting field of the meromorphic $T$-adic L-function is a uniformly finite extension.

The theorem is largely motivated by the conjectural behavior of the Coleman-Mazur eigencurve near the boundary of weight space \cite{buzzard-kilford}, which concerns a different tower of curves: Igusa tower.
Similar to that situation, we can also define an eigencurve parametrizing the zeros of various $L(\chi,s)$, if properly normalized.  
In our setup, we can prove the following
\begin{theorem}
The eigencurve constructed for the Artin-Schreier-Witt tower, when restricted to the boundary of the weight space, is an infinite union of subspaces which are finite and flat over the weight annulus, with slopes given by arithmetic progressions governed by the weight parameter.
\end{theorem}
We refer to Section~\ref{Sec:eigencurve} for a precise statement.
We feel that this theorem provides evidence for the case of Coleman-Mazur eigencurves.\footnote{Building on some key arguments of this paper, R.\ Liu and the second and third authors of this paper recently proved in \cite{liu-wan-xiao} similar results for the Coleman-Mazur eigencurves for definite quaternion algebras over $\QQ$.}
We however need to point out that the two cases are essentially different in several ways.
To name a few, (1) the Coleman-Mazur eigencurve corresponds to the Igusa tower where the Galois group is $\ZZ_p^\times$; 
(2) the relation between the analogous L-function and the characteristic function for the Coleman-Mazur eigencurve is twisted by power of cyclotomic characters; (3) the Coleman-Mazur eigencurve is expected to be very complicated near the center of the weight space, in contrast to the relatively simple characterization of Artin-Schreier-Witt eigencurve (see Theorem~\ref{T:eigencurve theorem}).

We end the introduction with a few questions.

\begin{question}
We restricted ourselves to the case when all coefficients of $\hat f(x)$ are Teichm\"uller lifts.
It would be interesting to know to what extent one can loosen this condition.
\end{question}

\begin{question}
We restricted ourselves to the case of $\ZZ_p$-extensions.  One can certainly slightly modify the defining equation for $C_m$ by changing $y_{\underline m}^F$ to $y_{\underline m}^{F^r}$ to consider a $\ZZ_{p^r}$-extension for some $r$.
It would be interesting to know if the analogous statements of our main theorems continue to hold in this case.
\end{question}

\begin{question}
It is also natural to consider an Artin-Schreier-Witt extension of higher dimensional tori, as in \cite{fu-wan}.  
In this case, the expected slopes in Theorem~\ref{T:main theorem} will be of a different form.
\end{question}

\begin{question}
We consider only Artin-Schreier-Witt towers in this paper.
It would be interesting to know if the Newton slopes of the L-functions for other $\ZZ_p$-towers of curves in the literature enjoy a similar periodicity property. 
\end{question}

\subsection*{Acknowledgements}
We thank Jun Zhang and Hui Zhu for interesting discussions.  The first author is partially supported by the Danish National Research Foundation through the Niels Bohr Professorship. The second author is partially supported by Simons Fellowship. 
The third author is partially supported by  Simons Collaboration Foundation \#278433, and  CORCL research grant from University of California, Irvine.

\section{$T$-adic exponential sums}
In this section, we recall properties of the L-function associated to a $T$-adic exponential sum as considered by C. Liu and the second author in \cite{fu-wan}. Its specializations to appropriate values of $T$ interpolate the L-functions considered above.

\begin{definition}
We use $T$ as an indeterminate and let $f$ and $\hat{f}$ be as in the introduction.
For a positive integer $k$, the \emph{$T$-adic exponential sum} of $f$ over $\FF_{q^k}^\times$ is the sum:
\[
S^*(k, T): = \sum_{x \in \FF_{q^k}^\times} (1+T)^{\Tr_{\QQ_{q^k} / \QQ_p}(\hat f(\hat x))} \in \ZZ_p\llbracket T\rrbracket.\footnote{This sum agrees with $S_f(k,T)$ in \cite{fu-wan} (in the one-dimensional case).}
\]
Note that the sum is taken over $\FF_{q^k}^\times$ as opposed to $\FF_{q^k}$. 
The $*$-notation simply reminds us that we are working over the torus $\GG_m$. 
The associated \emph{$T$-adic L-function} of $f$ over $\GG_{m, \FF_q}$ is the generating function
\begin{equation}
\label{E:L(T,s)}
L^*(T,s) = \exp \big( \sum_{k=1}^\infty S^*(k,T)\frac{s^k}{k} \big)
 = \prod_{x \in |\GG_m|}
\frac{1}{1-(1+T)^{\Tr_{\QQ_{q^{\deg(x)}}/\QQ_p}(\hat f(\hat x))}s^{\deg(x)}} \in 1+ s \ZZ_p\llbracket T \rrbracket \llbracket s\rrbracket.
\end{equation}
We put
\[
L(T,s) = \frac{L^*(T,s)}{ 1-(1+T)^{\Tr_{\QQ_q / \QQ_p}(\hat f(0))}s} = \prod_{x \in |\AAA^1|}
\frac{1}{1-(1+T)^{\Tr_{\QQ_{q^{\deg(x)}}/\QQ_p}(\hat f(\hat x))}s^{\deg(x)}}\in 1+ s\ZZ_p\llbracket T \rrbracket \llbracket s \rrbracket.
\footnote{Our $L^*(T,s)$ (but not $L(T,s)$) agrees with the $L_f(T,s)$ in \cite{fu-wan} (in the one-dimensional case).}
\]
Note that $L(T,s)$ (resp.\ $L^*(T,s)$) is the L-function over $\AAA^1$ (resp.\ $\GG_m$) associated to the character $\Gal(C_\infty/ \AAA^1) \cong \ZZ_p \to \ZZ_p\llbracket T \rrbracket^\times$ sending $1$ to $1+T$.
It is clear that for a finite character $\chi: \Zp \to \CC_p^\times$, we have $L(T,s)|_{T = \pi_\chi} = L(\chi, s)$ for $\pi_\chi = \chi(1) -1$.
\end{definition}

The $T$-adic L-function is a $T$-adic meromorphic function in $s$.  
It is often useful to consider a related holomorphic function in $s$ defined as follows:
\begin{definition}
The \emph{$T$-adic characteristic function} of $f$ over $\GG_{m,\FF_q}$ is the generating function
\[
C^*(T,s) = \exp \big( \sum_{k=1}^\infty \frac 1{1-q^k} S^*(k,T)\frac{s^k}{k} \big).
\footnote{Our $C^*(T,s)$ agrees with the $C_f(T,s)$ in \cite{fu-wan} (in the one-dimensional case); we will not introduce a version $C(T,s)$ without the star since it will not be used in our proof.}
\]
Clearly, we have 
$$C^*(T,s) = L^*(T,s) L^*(T,qs) L^*(T,q^2s) \cdots, \quad \textrm{and}\quad
L^*(T, s) = \frac{C^*(T,s)}{C^*(T,qs)}.$$
In particular, $C^*(T,s) \in 1+ s\ZZ_p
\llbracket T \rrbracket \llbracket s\rrbracket$.

Similarly, for a nontrivial finite  character $\chi: \ZZ_p \to \CC^\times_p$,
we put 
\[
L^*(\chi,s) = \big( 1 - \chi(\Tr_{\QQ_q/\QQ_p}(\hat f(0)))  s\big)
L(\chi, s).
\]
This is just the L-function of the character $\chi$ over $\GG_{m,\FF_q}$. 
The \emph{characteristic function} for $\chi$ is defined to be 
\[
C^*(\chi,s): = L^*(\chi,s) L^*(\chi,qs)L^*(\chi,q^2s) \cdots.
\]
It follows that 
\[
C^*(\chi,s) =C^*(T,s)|_{T = \pi_\chi} \quad \textrm{and} \quad L^*(\chi,s) = \frac{C^*(\chi,s)}{C^*(\chi,qs)}.
\]
\end{definition}

\begin{theorem}
\label{T:HP bound}The $T$-adic characteristic function $C^*(T, s)$ is $T$-adically entire in $s$.
Moreover, the $T^{a(p-1)}$-adic Newton polygon of $C^*(T,s)$ lies above the polygon whose slopes are
$0, \frac1d, \frac2d, \dots$, i.e., it lies above the polygon with vertices $(k, \frac{k(k-1)}{2d})$.
\end{theorem}
\begin{proof}
The first statement is \cite[Theorem~4.8]{fu-wan}, and the second one follows from the Hodge bound \cite[Theorem~5.2]{fu-wan}.
Note that the polyhedron $\Delta$ in {\it loc.~cit.}~is nothing but the interval $[0,d]$, and the function $W(k)$ in {\it loc.~cit.}~is the constant function $1$.
So $\mathrm{HP}_q(\Delta)$ in {\it loc.~cit.}~is nothing but the polygon written in the statement of our theorem, after being renormalized from a $T$-adic Newton polygon to a $T^{a(p-1)}$-adic Newton polygon, namely, after scaling vertically by $1/a(p-1)$.

For the convenience of the reader, we sketch the proof in the simpler case when $a=1$ (that is, $q=p$).
We first recall that the Artin-Hasse exponential series is defined as
\[
E(\pi) = \exp\big( \sum_{i=0}^\infty \frac{\pi^{p^i}}{p^i} \big) = \prod_{p \nmid i, i \geq 1} \big( 1-\pi^i\big)^{-\mu(i)/i} \in 1+ \pi + \pi^2 \ZZ_p\llbracket \pi \rrbracket,
\]
where $\mu(\cdot)$ is the M\"obius function.
We take $\pi$ to be the uniformizer in $\QQ_p\llbracket T \rrbracket$ such that $T = E(\pi) -1 = \pi +$ higher degree terms. Simple iteration calculation shows that $\pi \in T + T^2 \Zp\llbracket T\rrbracket$.

For our given polynomial $\hat f(x) = \sum_{i=0}^d \hat b_i x^i \in \ZZ_p[x]$, we put
\[
E_f(x) = \prod_{i=0}^d E(\hat b_i \pi x^i) \in \Zp\llbracket \pi \rrbracket \llbracket x \rrbracket.
\]
Dwork's splitting lemma \cite[Lemma~4.3]{fu-wan} (which can be checked by hand) says that for $x_0 \in \FF_{p^k}$, we have
\[
(1+T)^{\Tr_{\ZZ_{p^k}/ \Zp}(\hat f(\hat x_0))} = E(\pi)^{\Tr_{\ZZ_{p^k}/ \Zp}(\hat f(\hat x_0))} = \prod_{j=0}^{k-1} E_f(\hat{x}_0^{p^j}). 
\]

Let $B$ denote the Banach module over  $\ZZ_p\llbracket \pi^{1/d}\rrbracket$ with the formal basis $\Gamma=\{1, \pi^{1/d}x, \pi^{2/d}x^2, \dots\}$, that is, 
\[
B= \left\{
\sum_{i=0}^\infty c_i \pi^{i/d} x^i\; \big|\; c_i \in \ZZ_p\llbracket \pi^{1/d}\rrbracket 
\right\}.
\]
It is clear that we can write 
$$E_f(x)=\sum_{j=0}^\infty u_j \pi^{j/d} x^j \in B, \quad \textrm{for } u_j \in \ZZ_p.$$
Let $\psi_p$ denote the operator on $B$ defined by 
$$\psi_p(\sum_{i\geq 0}^\infty c_i x^i) = \sum_{i\geq 0}^\infty c_{pi} x^i.$$ 
Let $M$ be the matrix of 
the composite linear operator $\psi_p \circ E_f: B \to B$ with respect to the basis $\Gamma$, 
where $E_f$ is just multiplication by the power series $E_f(x)$.  One checks that 
\[
(\psi_p \circ E_f) \big( \pi^{i/d}x^i\big) = \sum_{j=0}^\infty u_{pj-i} \pi^{(p-1)j/d} \cdot (\pi^{j/d} x^j)
\]
Therefore, the infinite matrix $M=(u_{pj-i}\pi^{(p-1)j/d})_{0\leq i, j < \infty}$  has the shape 
\[
M=\begin{pmatrix}
* & * & * & \cdots \\
\pi^{\frac{p-1}{d}} * & 
\pi^{\frac{p-1}{d}} * &
\pi^{\frac{p-1}{d}} * & \cdots\\
\pi^{\frac{2(p-1)}{d}} * & 
\pi^{\frac{2(p-1)}{d}} * &
\pi^{\frac{2(p-1)}{d}} * & \cdots\\
\vdots & \vdots &\vdots & \ddots
\end{pmatrix}, 
\]
where $*$ denotes an element in $\ZZ_p\llbracket \pi^{1/d}\rrbracket$.  It follows that the matrix $M$ is nuclear. 

Now, one computes that 
$$(p-1)\Tr(M)=(p-1)\sum_{i=0}^{\infty} u_{(p-1)i}\pi^{(p-1)i/d}=\sum_{x\in \FF_p^{\times}}(\sum_{j=0}^{\infty} u_j\pi^{j/d}\hat{x}^j)=
\sum_{x\in \FF_p^{\times}}E_f(\hat{x}) = \sum_{x\in \FF_p^{\times}} (1+T)^{f(\hat{x})}.$$
Similarly, using the obvious property $G(x)\circ \psi_p = \psi_p \circ G(x^p)$ which holds for any power series $G(x)$, 
one checks  that for every positive integer $k$, 
$(p^k-1)\Tr(M^k)$ is given by 
$$ (p^k-1)\Tr((\psi_p\circ E_f)^k)=(p^k-1)\Tr(\psi_p^k\circ \prod_{i=0}^{k-1}E_f(x^{p^i}))=\sum_{x\in \FF_{p^k}^{\times}}
\prod_{i=0}^{k-1}E_f(\hat{x}^{p^i}) =S^*(k,T).$$
This is the additive form of the $T$-adic Dwork's trace formula. Its equivalent multiplicative form is 
\[
C^*(T, s) = \exp \big( \sum_{k=1}^\infty \frac 1{1-p^k} S^*(k,T)\frac{s^k}{k} \big)=  
\exp \big( \sum_{k=1}^\infty {-\Tr(M^k)}\frac{s^k}{k} \big) = \det \big( I - Ms).
\]
From the nuclear shape of $M$, it is clear that the $T^{(p-1)}$-adic (i.e., $\pi^{(p-1)}$-adic) Newton polygon of the characteristic 
power series of $M$ lies above the Hodge polygon of $M$, that is the polygon with slopes $0, \frac1d, \frac2d, \dots$.
\end{proof}

\section{Periodicity of Newton polygons}
Using the interplay of the $T^{a(p-1)}$-adic Newton polygon of $C^*(T, s)$ and the $\pi_{\chi}^{a(p-1)}$-adic Newton polygon of $L^*(\chi,s)$, 
one can deduce strong periodicity results.
From this, we recover results for $L(T,s)$ easily.

\begin{notation}
Let $C^*(T, s) = 1+ a_1(T)s + a_2 (T)s^2+ \cdots \in 1+ s\ZZ_p\llbracket T\rrbracket \llbracket s \rrbracket$ be the power series expansion in $s$.  Put $a_0=1$.
Then Theorem~\ref{T:HP bound} says that 
$$v_{T^{a(p-1)}}(a_n(T)) \geq \frac{n(n-1)}{2d}.$$
In other words, each $a_n(T)$ can be written as a power series in $T$:
\[
a_n(T) = a_{n, \lambda_n} T^{\lambda_n} + a_{n, \lambda_n+1} T^{\lambda_n+1}
+ a_{n, \lambda_n+2} T^{\lambda_n+2} +\cdots,
\]
with $a_{n,i} \in \ZZ_p$, $a_{n,\lambda_n} \neq 0$, and 
$$\lambda_n \geq \frac{n(n-1)a(p-1)}{2d}.$$
We call $a_{n, \lambda_n} T^{\lambda_n}$ the \emph{leading term} of $a_n(T)$, and $a_{n,\lambda_n}$ the \emph{leading coefficient} of $a_n(T)$.
\end{notation}

\begin{proposition}
\label{P:existence of turning point}
\emph{(1)}
For every nontrivial finite character $\chi:\ZZ_p \to \CC_p^\times$, the $\pi_\chi^{a(p-1)}$-adic Newton polygon of $C^*(\chi,s)$ lies above the polygon whose slopes are $0, \frac 1d, \frac 2d, \dots$, that is the polygon with vertices $(n, \frac{n(n-1)}{2d})$ for $n \in \ZZ_{\geq 0}$.

\emph{(2)} 
Let $\chi_1$ denote a nontrivial character of $\ZZ_p$ factoring through the quotient $\ZZ_p/p\ZZ_p$.
The $q$-adic Newton polygon of $C^*(\chi_1,s)$ lies below the polygon starting at $(0,0)$, and then with a segment of slope 0, then $d-1$ segments of slope $\frac{1}{2}$, then a segment of slope $1$, then $d-1$ segments of slope $1 + \frac{1}{2}$, and so on.

\emph{(3)}
The $q$-adic Newton polygon of $C^*(\chi_1, s)$ contains the line segment connecting $\big(nd, \frac{n(nd-1)}{2}\big)$ and $\big(nd+1, \frac{n(nd-1)}{2}+n\big)$ for all $n \in \ZZ_{\geq 0}$.
\end{proposition}
\begin{proof}
Part (1) follows from Theorem~\ref{T:HP bound} and the simple fact that the $T^{a(p-1)}$-adic Newton polygon of $C^*(T,s)$ lies below the $\pi_\chi^{a(p-1)}$-adic Newton polygon of $C^*(T,s)|_{T=\pi_\chi} = C^*(\chi,s)$.

For (2), we first recall that 
$$C^*(\chi_1,s) = L^*(\chi_1,s) L^*(\chi_1, qs) L^*(\chi_1, q^2s) \cdots.$$
It suffices to show that the $q$-adic Newton polygon of $L^*(\chi_1, s)$ lies below the polygon starting at $(0,0)$ and with a segment of slope $0$, and then $d-1$ segments of slope $\frac{1}{2}$.
But this is clear, as $L^*(\chi_1,s)$ is the product of the factor $1-\chi(\Tr_{\QQ_q/\QQ_p}(\hat f(0))) s$, which has slope zero, with the L-function $L(\chi_1,s)$, whose $q$-adic Newton polygon starts at $(0,0)$ and ends at $(d-1, \frac{d-1}{2})$.

Part (3) follows from the combination of (1) and (2) because the upper bound and lower bound of the $q$-adic Newton polygon of $C^*(\chi_1,s)$ agree on the line segments  connecting $\big(nd, \frac{n(nd-1)}{2}\big)$ and $\big(nd+1, \frac{n(nd-1)}{2}+n\big)$ for all $n \in \ZZ_{\geq 0}$.
Indeed, the upper bound $y$-coordinate corresponding to $x=nd$ in (2) is $n\frac{d-1}{2} + \frac{n(n-1)}{2}d = \frac{n(nd-1)}{2}$.
\end{proof}

\begin{remark}
Proposition~\ref{P:existence of turning point}(3), concerning the agreement of the Newton polygon upper and lower bounds, is the key point of this paper.  See Figure~\ref{upper lower figure} for an illustration of the upper and lower bounds in the special case $d = 4$.
\end{remark}

\begin{figure} 
\includegraphics[width=3in]{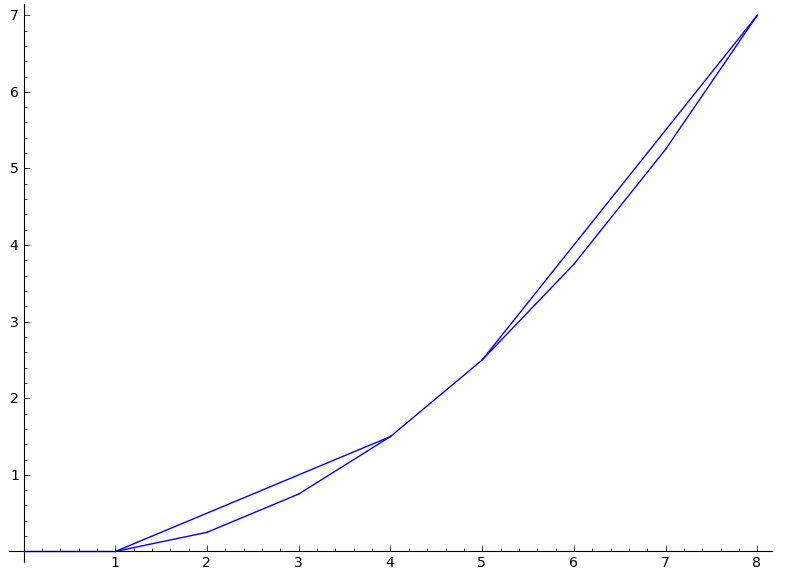}
\caption{The upper and lower bounds for the Newton polygon of the characteristic function over the interval $[0,8]$ for the special case $d = 4$.}\label{upper lower figure}
\end{figure}

\begin{proposition}
\label{P:ak units}
For an integer $k$ congruent to $0$ or $1$ modulo $d$, the leading term $a_{k,\lambda_k} T^{\lambda_k}$ of $a_k(T)$ satisfies $\lambda_k = \frac{k(k-1)a(p-1)}{2d}$ and $a_{k, \lambda_k} \in \ZZ_p^\times$.
\end{proposition}
\begin{proof}

From Proposition~\ref{P:existence of turning point}(3), we know the $q$-adic Newton polygon for $C^*(\chi_1, s)$ passes through the points $\left(nd,\frac{n(nd-1)}{2}\right)$ and $\left(nd+1, \frac{n(nd-1)}{2} + n\right)$. By considering the lower bound for the Newton polygon given by Proposition~\ref{P:existence of turning point}(1), the slope in the preceding segment must be less than or equal to 
$\frac{d-1}{d} + (n-1)$, 
and the slope in the following segment must be greater than or equal to $n + \frac{1}{d}$. 
This shows that these lattice points $\left(nd, \frac{n(nd-1)}{2}\right)$ and $\left(nd+1, \frac{n(nd-1)}{2} + n\right)$ must be 
vertices for both the $q$-adic Newton polygon of $C^*(\pi_{\chi_1}, s)$ and the $T^{a(p-1)}$-adic Newton polygon of $C^*(T,s)$. Hence we must have
\[
v_{\pi_{\chi_1}}(a_{nd}(\pi_{\chi_1})) = a(p-1)\frac{n(nd-1)}{2} = v_{T}(a_{nd}(T))\]
\[v_{\pi_{\chi_1}}(a_{nd+1}(\pi_{\chi_1})) = a(p-1)\left(\frac{n(nd-1)}{2}+ n\right) = v_{T}(a_{nd+1}(T)).  
\]
This shows that both $\lambda_{nd}$ and $\lambda_{nd+1}$ have the stated values and also shows that the leading coefficients of $a_{nd}(T)$ and $a_{nd+1}(T)$ are indeed $p$-adic units.
\end{proof}

\begin{corollary}
\label{C:strong upper bound}
For \emph{every} nontrivial finite  character $\chi:\ZZ_p \to \CC_p^\times$, the $\pi_\chi^{a(p-1)}$-adic Newton polygon of $C^*(\chi, s)$ contains the line segment connecting $\big(nd, \frac{n(nd-1)}{2}\big)$ and $\big(nd+1, \frac{n(nd-1)}{2}+n\big)$ for all $n \in \ZZ_{\geq 0}$.
Therefore, the $\pi_\chi^{a(p-1)}$-adic Newton polygon of $C^*(\chi,s)$ has the same upper bound polygon as described in Proposition~\ref{P:existence of turning point}(2).
\end{corollary}
\begin{proof}
By Proposition~\ref{P:ak units}, for $k =nd$ and $nd+1$, we have
\[
a_k(T) = a_{k, \lambda_k}T^{\lambda_k} + \text{ terms with higher power in $T$},
\] where $\lambda_k = \frac{k(k-1)a(p-1)}{2d}$ and $a_{k, \lambda_k} \in \ZZ_p^\times$.
Thus, we have
\[
v_{\pi_\chi^{a(p-1)}}\big(a_k(\pi_\chi)\big) = \frac{k(k-1)a(p-1)}{2d} \big/ (a(p-1)) = \frac{k(k-1)}{2d}.
\]
Note that this agrees with the lower bound of the $\pi_\chi^{a(p-1)}$-adic Newton polygon of $C^*(\chi,s)$ given by Proposition~\ref{P:existence of turning point}(1).
So it forces the  $\pi_\chi^{a(p-1)}$-adic Newton polygon of $C^*(\chi,s)$ to contain the segment connecting $\big(nd, \frac{n(nd-1)}{2}\big)$ and $\big(nd+1, \frac{n(nd-1)}{2}+n\big)$.
The last statement follows from the convexity of Newton polygons.
\end{proof}

\begin{corollary}
\label{C:T-adic line segment}
The $T^{a(p-1)}$-adic Newton polygon of $C^*(T,s)$ contains a line segment connecting $\big(nd, \frac{n(nd-1)}{2}\big)$ and $\big(nd+1,\frac{n(nd-1)}{2}+n\big)$ for all $n \in \ZZ_{\geq 0}$.
\end{corollary}
\begin{proof}
Note that the lower bound of the $T^{a(p-1)}$-adic Newton polygon of $C^*(T,s)$ given by Theorem~\ref{T:HP bound} is achieved at the points $x = nd$ or $nd+1$, as shown in Proposition~\ref{P:ak units}.
The corollary follows by convexity.
\end{proof}

\begin{lemma} \label{L:maximum gap} For every nontrivial finite  character $\chi: \ZZ_p \to \CC_p^\times$, 
the maximum gap between the lower bound and the upper bound for the $\pi_\chi^{a(p-1)}$-adic Newton polygon of $C^*(\chi,s)$, given in Proposition~\ref{P:existence of turning point}(1) and Corollary~\ref{C:strong upper bound},  is $\frac{(d-1)^2}{8d}$.
\end{lemma}

\begin{proof}
It suffices to consider the block above the $x$-axis interval $[1,d]$, because the upper and lower bounds in later intervals are obtained from the bounds in this interval by adding the same constant to both.  At the positive integer $x = i$ in the interval $[1,d]$, the gap is
$
g(i) = \frac{i-1}{2} - \frac{i(i-1)}{2d} = \frac{i(d+1) - d - i^2}{2d}.
$
The maximum of this function is 
$
g\left( \frac{d+1}{2} \right) = \frac{(d-1)^2}{8d}.
$
\end{proof}

\begin{theorem}
\label{T:NP independence}
For a finite nontrivial character $\chi: \ZZ_p \to \CC_p^\times$ with conductor $p^{m_\chi} \geq \frac{pa(d-1)^2}{8d}$, the $\pi_{\chi}^{a(p-1)}$-adic Newton polygon of $C^*(\chi,s)$ is independent of the character $\chi$.
\end{theorem}

\begin{proof}
By definition,  the $\pi_{\chi}^{a(p-1)}$-adic Newton polygon of $C^*(\chi,s)$ is the convex hull of points $\big(i, v_{\pi_\chi^{a(p-1)}}(a_i(\pi_\chi))\big)$ for all $i \geq 0$.
Since we have already given a (very strong) upper bound for this polygon in Corollary~\ref{C:strong upper bound}, we need only to consider those points which lies below this upper bound, which we will prove to be independent of the choice of $\chi$, provided that $p^{m_\chi }\geq \frac{pa(d-1)^2}{8d}$. We assume this inequality for the rest of the proof.

For each integer $i\geq 0$, write $a_i(T) = a_{i, \lambda_i}T^{\lambda_i} + a_{i, \lambda_i +1} T^{\lambda_i +1} + \cdots$ for $a_{i, j} \in \ZZ_p$ and $a_{i, \lambda_i} \neq 0$.
Then $\lambda_i \geq \frac{i(i-1)a(p-1)}{2d}$ by Theorem~\ref{T:HP bound}.
Let $\lambda'_i$ denote the minimal integer such that $a_{i, {\lambda'_i}}$ is a $p$-adic unit, or infinity if no such integer exists.
We claim the following:
\begin{itemize}
\item If either the point $\big(i, v_{\pi_\chi^{a(p-1)}}(a_i(\pi_\chi))\big)$ or the point $\big(i, \lambda'_i/a(p-1)\big)$ lies below the upper bound polygon  in Corollary~\ref{C:strong upper bound}, then $v_{\pi_\chi^{a(p-1)}}(a_i(\pi_\chi)) = \lambda'_i / a(p-1)$ (and the two points agree).
\end{itemize}
It would then follow that the $\pi_\chi^{a(p-1)}$-adic Newton polygon of $C^*(\chi, s)$ is the  lower convex hull of the upper bound polygon in Corollary~\ref{C:strong upper bound} and the points
\[
\big(i, \lambda'_i/a(p-1) \big) \textrm{ for each }i \geq 0.
\]
Note that the numbers $\lambda'_i$ are independent of $\chi$.  So the theorem follows from this claim.

We now prove the claim by studying $a_i(\pi_\chi) = a_{i, \lambda_i}\pi_\chi^{\lambda_i} + a_{i, \lambda_i +1} \pi_\chi^{\lambda_i +1} + \cdots$.
For $j \in [\lambda_i, \lambda'_i)$, the coefficient $a_{i,j}$ belongs to $\Zp$ and is divisible by $p$; so the valuation of the term is
\[
v_{\pi_\chi^{a(p-1)}}\big(a_{i,j} \pi_\chi^j \big) \geq  v_{\pi_\chi^{a(p-1)}}(p) + \frac{\lambda_i}{a(p-1)} = \frac{p^{m_\chi-1}}{a} + \frac{\lambda_i}{a(p-1)} \geq \frac{(d-1)^2}{8d} + \frac{\lambda_i}{a(p-1)},
\]
which corresponds to a point lying on or above the upper bound polygon by Lemma~\ref{L:maximum gap}.
In contrast, the term $a_{i,\lambda'_i} \pi_\chi^{\lambda'_i}$ has $\pi_\chi^{a(p-1)}$-valuation $\lambda'_i / a(p-1)$.
Therefore, if the point $\big(i, v_{\pi_\chi^{a(p-1)}}(a_i(\pi_\chi))\big)$ or $\big(i, \lambda'_i/a(p-1)\big)$ lies below the upper bound polygon, the $\pi_\chi^{a(p-1)}$-valuation of $a_i(\pi_\chi)$ must be just $\lambda'_i / a(p-1)$.  This proves the claim and hence the theorem.
\end{proof}

\begin{remark}
\label{R:same for non finite chars}
In fact, the same proof shows that, for any (not necessarily finite) character $\chi: \ZZ_p \to \CC_p^\times$ with $|\pi_\chi| = |\chi(1)-1| \geq p^{-8d/a(p-1)(d-1)^2}$, the $\pi_\chi^{a(p-1)}$-adic Newton polygon of $C^*(\chi,s)$ is independent of $\chi$.
\end{remark}

We are now ready to prove Theorem~\ref{T:main theorem}. 

\begin{proof}[Proof of Theorem{~\ref{T:main theorem}}]
Recall our setup: let $m_0$ be the minimal positive integer such that $p^{m_0-1} \geq \frac{a(d-1)^2}{8d}$ and let
$\chi_0: \ZZ_p \to \CC_p^\times$ be a finite character with $m_{\chi_0} = m_0$.
Let $0<\alpha_1, \dots, \alpha_{dp^{m_0-1}-1}<1$ denote the slopes of the $q$-adic Newton polygon of $L(\chi_0,s)$. 
Then $0, \alpha_1, \dots, \alpha_{dp^{m_0-1}-1}$ are the slopes of the $q$-adic Newton polygon of $L^*(\chi_0,s)$, and hence
\begin{equation}
\label{E:NP slopes for chi}
\bigcup_{i \geq 0} \big\{i,\alpha_1+i, \dots, \alpha_{dp^{m_0-1}-1} +i\big\}
\end{equation}
 are the slopes of the $q$-adic Newton polygon of $C^*(\chi_0,s)$. 
 Since $v(q)=a(p-1)p^{m_0-1}v(\pi_{\chi_0})$,  the slopes of the $\pi_{\chi_0}^{a(p-1)}$-adic Newton polygon of $C^*(\chi_0, s)$ are rescaled to 
 \begin{equation}
\label{E:NP slopes for chi_0}
\bigcup_{i \geq 0} \big\{p^{m_0-1}i, p^{m_0-1}(\alpha_1+i), \dots, p^{m_0-1}(\alpha_{dp^{m_0-1}-1} +i) \big\}. 
\end{equation}
Now given a finite  character $\chi: \ZZ_p \to \CC_p^\times$ with $m_\chi \geq m_0$,
Theorem~\ref{T:NP independence} says that the slopes of the  $\pi_\chi^{a(p-1)}$-adic Newton polygon of $C^*(\chi, s)$ are also given by \eqref{E:NP slopes for chi_0}.
Now $v(q)=a(p-1)p^{m-1}v(\pi_{\chi})$,  the slopes of the $q$-adic Newton polygon of $C^*(\chi, s)$ are rescaled to 
 \begin{equation}
\label{E:NP slopes for chi_2}
\bigcup_{i \geq 0} \big\{p^{m_0-m}i, p^{m_0-m}(\alpha_1+i), \dots, p^{m_0-m}(\alpha_{dp^{m_0-1}-1} +i) \big\}. 
\end{equation}
Using the relation
\[
L(\chi,s) = \frac{1}{1 - \chi(\Tr_{\QQ_q/\QQ_p}(\hat f(0)))  s} \cdot \frac{C^*(\chi,s)}{C^*(\chi, qs)},
\]
it is clear that the slopes of the $q$-adic Newton polygon of $L(\chi,s)$ are as described in Theorem~\ref{T:main theorem}.
\end{proof}

\begin{remark}
It would be interesting to know whether the slopes of the $T^{a(p-1)}$-adic Newton polygon of $C^*(T,s)$ satisfy a similar periodicity property, i.e., whether the slopes form a disjoint union of finite number of arithmetic progressions. This is known to be true if $p\equiv 1 \bmod d$ (see the references given in Example~\ref{ordinary example}), but open in general. 
\end{remark}

We give an application. Recall that the $T$-adic L-function $L(T,s)$ is a $T$-adic meromorphic (but in general not rational) function in $s$
over the field $\QQ_p((T))$. Thus, adjoining all the infinitely many zeros and poles of $L(T,s)$ to 
the field $\QQ_p((T))$ would generally  give an infinite extension. The following result shows this is not the case. 

\begin{theorem}
The splitting field $K_{p,d}$ over $\QQ_p((T))$ of all $T$-adic L-functions $L_f(T,s)$\footnote{Here, the subscript $f$ emphasizes the dependence of the L-function on the polynomial $f$.} for \emph{all} monic polynomials $f\in \overline {\FF}_p[x]$ of degree $d$  is a finite extension of $\QQ_p((T))$.
\end{theorem}
\begin{proof}
It suffices to prove that the splitting field $K'_{p,d}$ over $\QQ_p((T))$ of all $T$-adic characteristic functions $C^*(T,s)$ 
for all monic polynomials $f\in \overline{\FF}_p[x]$ of degree $d$  is a finite extension of $\QQ_p((T))$.
By Corollary~\ref{C:T-adic line segment}, 
the power series $C^*(T,s)$ factors as an infinite product of polynomials (in the variable $s$) of degree $\leq d$. So
the splitting field $K'_{p,d}$ is contained in the extension $\widetilde  K_{p,d}$ of $\QQ_p((T))$ given by adjoining zeros of all irreducible polynomials of degree $\leq d$.  But this is a finite extension of $\Qp((T))$ as we prove now.

First, since $\Qp$ has characteristic zero, all extensions of $\Qp((T))$ are tamely ramified. If we only adjoin zeros of irreducible polynomials of degree $\leq d$, then $\widetilde K_{p,d}$ is an extension of $\Qp((T))$ with $T$-ramification degree $\leq d!$.
It then suffices to bound the residual field extension, namely, to prove that the composite of all finite Galois extensions of $\Qp$ of degree $\leq d!$ is still a finite extension.  In fact, we will show that there are only finitely many Galois extensions of $\Qp$ with degree $\leq d!$.
For this, we first notice that there are only finitely many choices of Galois groups with order $\leq d!$, which are all solvable.  So we just need to prove that there are only finitely many \emph{abelian} extensions at each step.  But this is clear from local class field theory.
\end{proof}
\begin{remark}
One may be able to give a more precise bound on the extension degree of $K_{p,d}$ over $\QQ_p((T))$.  We leave this to interested readers.
\end{remark}

\begin{remark} 
The above result proves a strong form of the $T$-adic Riemann hypothesis for the $T$-adic L-function in \cite{fu-wan} in the sense of Goss \cite{Goss}, see 
\cite{wan1} and \cite{Sheats} for evidence for Goss's original conjecture for his characteristic $p$ zeta functions.  For Dwork's unit root zeta function 
which is known to be a $p$-adic meromorphic function, the corresponding $p$-adic Riemann hypothesis is essentially completely open, even for the weaker version about the finiteness of 
the ramification of the splitting field over $\QQ_p$; see Conjecture 1.3 in \cite{wan2}.  For the characteristic power series 
of the $U_p$-operator acting on the $p$-adic Banach space of overconvergent $p$-adic modular forms of a given level and weight, the finiteness of the ramification of  the splitting field over $\QQ_p$ is also unknown; see Conjecture 6.1 in \cite{wan4}. An example of the $p$-adic Riemann hypothesis for 
zeta functions of divisors is given in \cite{wh}. 
\end{remark}

\begin{remark} 
Another natural problem is to study the possible simplicity of the zeros of the $T$-adic characteristic series $C^*(T,s)$. 
The simplicity is known in the case $p\equiv 1\pmod d$. It would be interesting to know if simplicity remains true in general.
\end{remark}

A similar proof gives the following more classical application. 
\begin{theorem}
Let $E_{p,d}(m)$ be the splitting field over $\QQ_p$ of all zeta functions $Z(C_m,s)$ for \emph{all} monic polynomials $f\in \overline{\FF}_p[x]$ of degree $d$. 
Then there is an explicit constant $B_d$ depending only on $d$ such that for all $m\geq 1$, we have 
$$[E_{p,d}(m): \QQ_p] \leq B_d p^{m-1}.$$
Furthermore, the inequality is an equality with $B_d = d$ if $p \equiv 1 \pmod d$. 
\end{theorem}

\section{Eigencurves for Artin-Schreier-Witt towers} \label{Sec:eigencurve}

For the Igusa tower over the ordinary locus of the modular curves, Coleman and Mazur \cite{coleman-mazur} studied a certain analogous $T$-adic characteristic function; from this, they defined  an eigencurve parametrizing the zeros of the characteristic function.
This eigencurve has many applications in number theory.

One of the striking results about the Coleman-Mazur eigencurve is its nice behavior near the boundary of the weight space, as shown by K. Buzzard and L. Kilford in \cite{buzzard-kilford} when $p=2$.
Unfortunately, such results are only known for very small prime numbers $p$.\footnote{Building on some key techniques developed in this paper, R. Liu, J. Zhang and the second and the third author recently proved a large range of cases for Coleman-Mazur eigencurves; see \cite{liu-wan-xiao} and 
\cite{xiao-zhang}.}

In this section, we study the analogous construction for the Artin-Schreier-Witt tower of curves, and we prove strong geometric properties of the analogous eigencurve near the boundary of the weight space.

\begin{definition}
Let $\calW$ denote the the rigid analytic open unit disc associated to $\ZZ_p\llbracket T \rrbracket$.
The \emph{eigencurve} $\calC_f$ associated to the Artin-Schreier-Witt tower for $f(x)$ is defined to be the zero locus of $C^*(T,s)$, viewed as a rigid analytic subspace of $\calW \times \GG_{m, \rig}$, where $s$ is the coordinate of the second factor. 
Denote the natural projection to the first factor by $\mathrm{wt}: \calC_f \to \calW$; and denote the \emph{inverse} of the natural projection to the second factor by 
\[
\boldsymbol \alpha: \calC_f \xrightarrow{\mathrm{pr}_2} \GG_{m, \rig} \xrightarrow{x \mapsto x^{-1}} \GG_{m, \rig}.
\]
For a closed point $w$ on $\calW$,
we use $v_{\calW}(w)$ to denote the $p$-adic valuation of the $T$-coordinate  of $w$.  Similarly, for a closed point $z \in \GG_{m, \rig}$, we use $v_{\GG_m}(z)$ to denote the $p$-adic valuation of the $s$-coordinate of $z$. 

Note that, by Theorem~\ref{T:HP bound} and Proposition~\ref{P:ak units}, $C^*(0,s) = 1+ a_1(0)s$ with $a_1(0)$ a $p$-adic unit. So over the point $T=0$ of $\calW$, $\calC_f$ has only one point and most components of $\calC_f$ blow up as $T$ approaches to $0$. This is slightly different from the case of usual Coleman-Mazur eigencurve.
For this reason, we put $\calW^\circ = \calW\backslash \{0\}$ as a rigid space, and $\calC_f^\circ: = \wt^{-1}(\calW^\circ)$.
\end{definition}

\begin{theorem}
\label{T:eigencurve theorem}
The following properties hold for the eigencurve $\calC_f$.
\begin{itemize}
\item[(1)]
The formal power series
$C^*(T,s)$ can be written as an infinite product  $\prod_{i=0}^\infty P_i(s)$, where each polynomial $P_i(s) = 1 + b_{i,1}(T)s + \cdots + b_{i,d}(T)s^d$ belongs to $\ZZ_p\llbracket T\rrbracket[s]$, whose $T^{a(p-1)}$-adic Newton polygon accounts for the segment between $x \in [(i-1)d, id-1]$ of the $T^{a(p-1)}$-adic Newton polygon of $C^*(T,s)$, and the leading term of $b_{i,d}(T)$ has coefficients in $\ZZ_p^\times$.

\item[(2)]  The eigencurve $\calC_f^\circ$ is an infinite disjoint union $\coprod_{i \geq 0} \calC_{f,i}^\circ$, where 
each $\calC_{f,i}^\circ$ is the zero locus of the polynomial $P_i(s)$ and it is
 a finite and flat cover of $\calW^\circ$ of degree $d$. 

\item[(3)]
Put $r = p^{-8d/a(p-1)(d-1)^2}$.
Let $\calW^{\geq r}$ denote the annulus inside $\calW$ where $|T| \geq r$.
Then there exist an integer $l \in \NN$ and (distinct) rational numbers $\beta_1, \dots, \beta_l \in [0,1)$ such that each $\calC^\circ_{f, i} \times_\calW \calW^{\geq r}$ is a disjoint union $\coprod_{j=1}^l \calC_{f, i}^{\circ,(j)}$ of closed subspaces of $\calC^\circ_{f,i}$, each being finite and flat over $\calW^{\geq r}$, and is characterized by the following property:
\[
\forall z \in \calC_{f,i}^{\circ,(j)},\quad v_{\GG_m}\big(\boldsymbol \alpha(z)\big) = ap^{m_0-1}(p-1)(\beta_j+i) v_{\calW}\big(\wt(z)\big).
\]
\end{itemize}
\end{theorem}
\begin{proof}
The decomposition in (1) follows from the basic fact on the relation between Newton polygons and factorizations, in light of Corollary~\ref{C:T-adic line segment}. Moreover, by Proposition~\ref{P:ak units}, all $b_{i,j}(T)$ has coefficients in $\ZZ_p$ and the leading term of $b_{i,d}(T)$ is a $p$-adic unit.

Having the factorization at hand, it is clear that $\calC_f^\circ$ is the union of the zero loci of the $P_i(s)$'s, which are closed analytic subspaces $\calC^\circ_{f,i}$ of $\calW^\circ \times \GG_{m, \rig}$.
Moreover, since for any character $\chi: \ZZ_p \to \CC_p^\times$ with $\pi_\chi:=\chi(1)-1$, the slopes of the $\pi_\chi^{a(p-1)}$-adic Newton polygons of  $P_i(s)|_{T = \pi_\chi}$
sits in $[i-1, i)$.  So the zeros of $P_i(s)|_{T=\pi_\chi}$ are distinct for different $i$.
This implies that all subspaces $\calC^\circ_{f, i}$ are disjoint and concludes the proof of (2).

For (3), let $\chi_0: \ZZ_p \to \CC_p^\times$ be a finite  character with conductor $ \geq p \frac{a(d-1)^2}{8d}$.
Let $0, \alpha_1, \dots ,\alpha_{dp^{m_0-1}-1}$ be the $q$-adic slopes in $ L^*(\chi_0,s)$ (counted with multiplicities). These are rational numbers in the 
interval $[0,1)$. 
Each point $w \in \calW^\circ$ gives rise to a  (not necessarily finite) character $\chi_w: \Zp \to \CC_p^\times$; put $\pi_{\chi_w} = \chi_w(1)-1$.
By Remark~\ref{R:same for non finite chars},  the slopes of the $\pi_{\chi_w}^{a(p-1)}$-adic Newton polygon of $C^*(T,s)|_{T = \pi_{\chi_w}}$ are exactly 
\begin{equation}
\label{E:slopes for C at w}
\bigcup_{i \in \ZZ_{\geq 0}}
\big\{
p^{m_0-1}i, p^{m_0-1}(\alpha_1 + i), \dots, p^{m_0-1}(\alpha_{dp^{m_0-1}-1} + i)
\big\}.
\end{equation}
Then the slopes of the $p$-adic Newton polygon of $C^*(T,s)|_{T=\pi_{\chi_w}}$ should be given by the numbers in  \eqref{E:slopes for C at w} times the normalizing factor 
$
a(p-1)v_p(\pi_{\chi_w}) = a(p-1)v_\calW(w).
$

Let $0=\beta_1< \cdots < \beta_l<1$ be the slopes $0,\alpha_1, \dots, \alpha_{dp^{m_0-1}-1}$ with repeated numbers removed.  
Then the slope information above implies that the intersections $\calC_{f,i}^{\circ,(j)}$ of $\calC^\circ_{f,i}$ with the subdomain
\[
\big\{ z \in \calW^{\geq r} \times \GG_{m, \rig}\;
\big|\;
v_{\GG_m}\big(\boldsymbol \alpha(z)\big) = ap^{m_0-1}(p-1)(\beta_j+i) v_{\calW}\big(\wt(z)\big)
\big\}
\]
form a finite cover of $\calC^\circ_{f,i} \times_{\calW^\circ} \calW^{\geq r}$ by affinoid subdomains.
Thus, the union  
$$\calC^\circ_{f,i} \times_{\calW^\circ }\calW^{\geq r} =\bigsqcup_{j=1}^l \calC_{f,i}^{\circ,(j)}$$ 
is a disjoint union.
Each  $\calC_{f,i}^{\circ,(j)}$ is finite and flat over $\calW^{\geq r}$ because its fiber over every point of $\calW^{\geq r}$ is exactly the multiplicity of $\beta_j$ in the collection of $\alpha$'s above.
The assertions in (3) are now proved.
\end{proof}

\begin{remark}
The analogous statement of Theorem~\ref{T:eigencurve theorem}(2) for Coleman-Mazur eigencurve is probably too strong to be true.  But it is generally believed that, at least under certain conditions, the analogous statement of Theorem~\ref{T:eigencurve theorem}(3) for Coleman-Mazur eigencurve holds (as suggested by Buzzard-Kilford \cite{buzzard-kilford}). 
\end{remark}


\begin{thebibliography}{9999}
\bibitem[BK]{buzzard-kilford}
K. Buzzard and L. Kilford, 
The 2-adic eigencurve at the boundary of weight space,
{\it Compos. Math.} {\bf 141} (2005), no. 3, 605--619. 

\bibitem[CM]{coleman-mazur}
R. Coleman and B. Mazur,
The eigencurve, in {\it Galois representations in arithmetic algebraic geometry (Durham, 1996)}, 1--113, 
{\it London Math. Soc. Lecture Note Ser.}, {\bf 254}, Cambridge Univ. Press, Cambridge, 1998. 

\bibitem[Go]{Goss}
D. Goss, A Riemann hypothesis for characteristic $p$ L-functions,  
{\it J. Number Theory}  {\bf 82} (2000), no. 2, 299--322. 


\bibitem[HW]{wh} C. Haessig and D. Wan, 
On the $p$-adic Riemann hypothesis for the zeta function of divisors, 
{\it J. Number Theory}  {\bf 104} (2004), no. 2, 335--352.

\bibitem[LWan]{fu-wan}
C.  Liu and D. Wan,
$T$-adic exponential sums over finite fields, {\it Algebra Number Theory} {\bf3} (2009), no. 5, 489--509. 


\bibitem[LWei]{liu-wei} C. Liu and D. Wei, 
The L-functions of Witt coverings, {\it Math. Z.}  {\bf 255} (2007), 95--115. 

\bibitem[LWX]{liu-wan-xiao}
R. Liu, D. Wan, and L. Xiao,
Slopes of eigencurves over the boundary of the weight space, {\tt arXiv:1412.2584}.

\bibitem[Sh]{Sheats}   
 J. Sheats, The Riemann hypothesis for the Goss zeta function for ${\FF_q}[T]$, 
 {\it J. Number Theory} {\bf 71} (1998), no. 1, 121--157. 

\bibitem[W1]{wan1}
D. Wan, On the Riemann hypothesis for the characteristic $p$ zeta function, 
{\it J. Number Theory} {\bf 58} (1996), no. 1, 196--212.

\bibitem[W2]{wan4}
D. Wan, Dimension variation of classical and $p$-adic modular forms, 
{\it Invent. Math.} {\bf 133} (1998), no. 2, 449--463. 

\bibitem[W3]{wan2}
D. Wan, Dwork's conjecture on unit root zeta functions, 
{\it Ann.  Math.} {\bf 150} (1999), no. 3, 867--927. 

\bibitem[W4]{wan3}
D. Wan, Variation of $p$-adic Newton polygons for L-functions of exponential sums, 
{\it Asian J. Math.} {\bf  8} (2004), no. 3, 427--471.

 
\bibitem[WXZ]{xiao-zhang}
D. Wan, L. Xiao and J. Zhang,
Slopes of eigencurves over boundary disks, {\tt  arXiv:1407.0279}.


\end{thebibliography}
\end{document}